\documentclass[11pt,a4paper]{amsart}%
\usepackage{amsmath}
\usepackage{amsfonts}
\usepackage{amssymb}
\usepackage{hyperref}
\usepackage{enumerate,color}
\usepackage{graphicx}%
\usepackage[ngerman,french,english]{babel}
\usepackage[T1]{fontenc}
\providecommand{\U}[1]{\protect\rule{.1in}{.1in}}
\newtheorem{theorem}{Theorem}[section]
\newtheorem{thm}[theorem]{Theorem}

\newtheorem{corollary}[theorem]{Corollary}
\newtheorem{cor}[theorem]{Corollary}

\newtheorem{defn}[theorem]{Definition}
\newtheorem{example}[theorem]{Example}

\newtheorem{lemma}[theorem]{Lemma}
\newtheorem{lem}[theorem]{Lemma}

\newtheorem{proposition}[theorem]{Proposition}
\newtheorem{remark}[theorem]{Remark}

\newcommand{\C}{\mathbb{C}}
\newcommand{\cO}{\mathcal{O}}
\newcommand{\eps}{\varepsilon}
\newcommand{\R}{\mathbb{R}}
\newcommand{\RR}{\mathbb{R}}
\newcommand{\CC}{\mathbb{C}}

\newcommand{\N}{\mathbb{N}}

\newcommand{\diag}{\operatorname{diag}}

\newcommand{\GL}{\operatorname{GL}}

\def\eps{\varepsilon}
\def\alp{\alpha}

\newcommand{\Dima}[1]{{{#1}}}
\newcommand{\DimaA}[1]{{{#1}}}

\begin{document}

\author{Dmitry Gourevitch}
\address{The Incumbent of Dr. A. Edward Friedmann Career Development Chair in Mathematics, Department of Mathematics, Weizmann
Institute of Science, POB 26, Rehovot 76100, Israel }
\email{dimagur@weizmann.ac.il}
\urladdr{\url{http://www.wisdom.weizmann.ac.il/~dimagur}}


\date{\today}
\title{Composition series for degenerate principal series of $\mathrm{GL}(n)$ }

\begin{abstract}
In this note we consider representations of the group $GL(n,F)$, where $F$ is the field of real or complex numbers
or, more generally, an arbitrary local field, in the space of equivariant line bundles over Grassmannians over the same field $F$.
We study reducibility and composition series of such representations.

Similar results were obtained already in \cite{HL,alesker-alpha,Zel}, but we give a short uniform proof in the general case, using the tools from \cite{AGS}. We also indicate some applications to cosine transforms in integral geometry.


\end{abstract}

\maketitle

\begin{otherlanguage}{french}
\begin{abstract}
Dans cette note on consid\`{e}re des repr\'{e}sentations du groupe $GL(n,F)$, o\`{u} $F$ est le corps des nombres r\'{e}els ou complexes ou plus g\'{e}n\'{e}ralement,
un corps local arbitraire, dans l'espace de fibres en droites \'{e}quivariants sur des Grassmanniennes sur le m\^{e}me corps $F$.
On \'{e}tudie la r\'{e}ductibilit\'{e} et la suite de composition de telles repr\'{e}sentations.

Les r\'{e}sultats similaires ont d\'{e}j\`{a} \'{e}t\'{e} obtenus dans \cite{HL,alesker-alpha,Zel}, mais nous
pr\'{e}sentons une courte preuve dans le cas g\'{e}n\'{e}ral en utilisant les outils de \cite{AGS}.
On donne aussi quelques applications aux transform\'{e}es en cosinus en g\'{e}om\'{e}trie int\'{e}grale.

\end{abstract}
\end{otherlanguage}

\section{Introduction}
Let $F$ be a local field, {\it i.e.} either $\R$ or $\C$ or a finite extension of the field $\mathbb{Q}_p$ of $p$-adic numbers for some prime $p$, or the field $\mathbb{F}_q((t))$ of Laurent series over a finite field.
Let $G=G_{n}:=GL(n,F)$ be the group of invertible matrices of size $n$ with entries in $F$.

Then $C^{\infty}\left(
G\right)  $ is a $G\times G$-module with left and right actions given by
$L_{g}f\left(  x\right)  =f\left(  g^{-1}x\right)  $, $R_{g}f\left(  x\right)
=f\left(  xg\right)  $. Let $P\subset G$ be an \DimaA{algebraic}  subgroup with modular
function $\Delta_{P}$ and let $\chi$ be a character of $P$. The induced
representation $I\left(  P,\chi\right):=Ind_{P}^{G}(\chi)  $ is the right $G$-action on the space%
\begin{equation}
C^{\infty}\left(  G,P,\chi\right)  :=\left\{  f\in C^{\infty}\left(  G\right)
\mid L_{p^{-1}}f=\chi\left(  p\right)  \Delta_{P}^{1/2}\left(  p\right)
f\text{ for all }p\in P\right\}, \label{=CG}%
\end{equation}
whose elements may also be regarded as smooth sections of a line bundle on
$G/P$.
\DimaA{If $F$ is not $\mathbb{R}$ and not $\mathbb{C}$ then smooth means locally constant.}
We are  interested in the so-called maximal parabolic
subgroups $P=P_{p_{1},p_{2}}$, with $p_{1}+p_{2}=n$, consisting of matrices
$x\in G_{n}$ of the form
\begin{equation}
x=\left[
\begin{array}
[c]{ll}%
x_{11} & x_{12}\\
0 & x_{22}%
\end{array}
\right]  ;x_{ij}\in Mat_{p_{i}\times p_{j}}\text{.} \label{=x}%
\end{equation}
In this case $G/P$ is the Grassmannian of $p_{1}$-dimensional subspaces of
$F^{n}$. The characters of $P$ are of the form $\chi_{1}\otimes\chi_{2}\left(
x\right)  =\chi_{1}\left(  x_{11}\right)  \chi_{2}\left(  x_{22}\right)  $,
where $\chi_{i}$ is a character of $G_{p_{i}}$. Following \cite{BZ-Induced} we
write $\chi_{1}\times\chi_{2}$ instead of $I\left(  P,\chi_{1}\otimes\chi
_{2}\right)  $.

It is well known that the representation $\chi_{1}\times\chi_{2}$ is irreducible for most characters $\chi_1,\chi_2$. We explicitly describe the exceptional characters for which  $\chi_{1}\times\chi_{2}$ is reducible, and show that the composition series are monotone in a certain sense. We also show that if $F$ is non-archimedean, the length of $\chi_{1}\times\chi_{2}$ is at most two, {\it i.e.} it has at most one non-trivial $G$-invariant subspace.

In \S \ref{subsec:main} below we describe the main results precisely. In \S \ref{subsec:IntGeo}
below we give applications to integral geometry.

Most of the results are known, see \S \ref{subsec:known}. However, we give short proofs that work uniformly for all fields. Our main tool\DimaA{s} are the Bernstein-Zelevinsky derivatives (\cite{BZ-Induced}) and their archimedean analogs developed in \cite{AGS,AGS2}. We describe these, as well as other preliminaries, in \S \ref{sec:prel}.

\subsection{Main results}\label{subsec:main}
We fix some notation to describe our main results succinctly. For $z\in F$ let
$\nu\left(  z\right)  $ denote the positive scalar by which the additive Haar
measure on $F$ transforms under multiplication by $z$. We will also regard
$\nu$ as a character of $G_{k}$ defined by $\nu(g):=\nu_k(g):=\nu(\det g)$, and we note
that the modular function of $P=P_{p_{1},p_{2}}$ is $\Delta_{P}=\nu^{p_{2}%
}\otimes\nu^{-p_{1}}$. \Dima{This means that for $x$ as in  \eqref{=x} we have $\Delta_{P}(x)=\DimaA{\nu(x_{11})^{p_2}\nu(x_{22})^{-p_1}}$.}
We also let $1$ denote the trivial character of $G_k$, $\eps$ denote the
 $sgn\left(  \det\right)  $ character of $\GL_{k}\left(
\mathbb{R}\right),$ and $\alpha$ denote the character $\det |\det|^{-1}$ of  $\GL_{k}\left(
\mathbb{C}\right).$ Twisting by a character of $G_n$ we can assume $\chi_2=1$ and consider $\chi \times 1$. Let $r:=\min(p_1,p_2)$. \Dima{In \S \ref{sec:Pf} we prove the following statements.}

\begin{thm}
\label{thm:irr} $\chi \times 1$ is reducible if and only if one of
the following holds

\begin{enumerate}
[(i)]

\item \label{it:irrmixed} $\chi=\nu^{\pm(k-n/2)}$, where $k$ is an
integer, and $0\leq k\leq r-1$.

\item \label{it:irrR}$F=\mathbb{R}$, and $\chi=\eps^{k+1}\nu^{\pm(k-r+n/2})$ for some positive integer $k$.

\item \label{it:irrR2}$F=\mathbb{R}, \, r>1$ and
$\chi=\nu^{\pm(k-r+n/2+1)}$ or $\chi=\eps\nu^{\pm(k-r+n/2+1)}$ for some positive integer $k$.

\item \label{it:irrC}$F=\mathbb{C}$, and either $\chi$ or $\chi^{-1}$ equals $\alpha^k\nu^s$ where $s-n/2$ and $k$ are integers and $s+k,s-k>n/2-r.$
\end{enumerate}
\end{thm}

We also show (Lemma \ref{lem:padic}) that  if $F$ is non-archimedean then $length(\chi \times 1) \leq 2$.

\Dima{We} then show that $\chi\times 1$ has a unique composition
series, monotone in some sense. For Archimedean $F$ we could use  the Gelfand - Kirillov dimension. For the general case we will use the notion of rank, see \S\ref{subsec:rank}. Define $s(\chi)\in \R$ by $|\chi|=\nu^{s(\chi)}$.

\begin{thm}
\label{thm:mon} If $s(\chi)>0$
 then the irreducible constituents of $\chi \times 1$ appear in  strictly descending order
of rank, {\it i.e.} \Dima{for any subrepresentaion $\sigma$ of $\chi \times 1$,
the quotient $(\chi \times 1)/\sigma$ has a unique irreducible subrepresentaion, and its rank is higher than the the ranks of all other irreducible Jordan-Holder constituents of $(\chi \times 1)/\sigma$.

If $s(\chi)\leq0$ then the irreducible constituents of $\chi \times 1$ appear  in   strictly ascending order of rank.}
\end{thm}

Consider the standard intertwining operator $I_{\chi}:\chi\times 1 \to 1\times \chi$. From Theorem \ref{thm:mon} we obtain
the following corollary. \begin{cor}\label{cor:IntOp}
The image of $I_{\chi}$ is the only irreducible quotient of  $\chi\times 1$ and the only irreducible submodule of $1\times \chi$. In particular,  $I_{\chi}$ is invertible if and only if $\chi \times 1$ is irreducible.
\end{cor}

\subsection{Connections to integral geometry}\label{subsec:IntGeo}

The $\alp$-cosine transform on real Grassmannian in a Euclidean space $\RR^n$, including the case of even functions on the sphere,
was studied by both analysts and geometers for a long period of time: \cite{semjanistyi,schneider,koldobsky,ournycheva-rubin, rubin-inversion,rubin-adv,rubin-funk-cosine}.
The case $\alp=1$ plays a special role in convex and stochastic geometry, see \cite{goodey-howard1,goodey-howard2,goodey-howard-reeder}, and in particular in valuation theory
\cite{AB,alesker-jdg-03,alesker-gafa-lefschetz,alesker-fourier}. One of the reasons why the case of at least some values $\alpha\ne 1$ is relevant for geometry is the recent
result \cite{AlGS} saying that the Radon transform between the Grassmannians of $i$- and $(n-i)$-dimensional linear subspaces in $\RR^n$ coincides with the
$\alpha=-\min\{i,n-i\}$-cosine transform\footnote{For $i=1$ this result is known for a long time and seems to be a folklore.}, and moreover sometimes composition of two Radon transforms coincides with the $\alpha$-cosine transform for appropriate
$\alpha\ne 1$.

Let us recall the definition of the $\alpha$-cosine transform. Let $Gr_i(\RR^n)$ denote the Grassmannian
of $i$-dimensional linear subspaces in the Euclidean space $\RR^n$. For a pair of subspaces $E,F\in Gr_i(\RR^n)$ one defines (the absolute value of)
the cosine of the angle $|\cos(E,F)|$ between $E$ and $F$ as the coefficient of the distortion of measure under
the orthogonal projection from $E$ to $F$. More precisely, let $q\colon E\to F$ denote the restriction to $E$ of
the orthogonal projection $\RR^n\to F$. Let $A\subset E$ be an arbitrary subset of finite positive Lebesgue measure.
Then define
$$|\cos(E,F)|:=\frac{vol_F(q(A))}{vol_E(A)},$$
where $vol_F, vol_E$ are Lebesgue measures on $F,E$ respectively induced by the Euclidean metric on $\RR^n$ normalized
so that the Lebesgue measures of unit cubes are equal to 1. It is easy to see that $|\cos(E,F)|$
is independent of the set $A$ and is symmetric with respect to $E$ and $F$.

\begin{example}
If $i=1$ then $|\cos(E,F)|$ is the usual cosine of the angle between two lines.
\end{example}

For $\alp\in \CC$ with $Re(\alp)\geq 0$ let us define the $\alpha$-cosine transform
$$T_\alp:=T^i_\alp\colon C^\infty(Gr_i(\RR^n))\to C^\infty(Gr_i(\RR^n))$$
by
\begin{eqnarray}\label{E:Int1}
(T^{i}_\alp f)(E)=\int_{F\in Gr_i(\RR^n)}|\cos(E,F)|^\alp f(F) dF,
\end{eqnarray}
where $dF$ is the $O(n)$-invariant Haar probability measure on $Gr_i(\RR^n)$.
\begin{remark}\label{Int-R:2}
One can show that the integral (\ref{E:Int1}) absolutely converges for $Re(\alp)>-1$ (see  \cite[Lemma 2.1]{alesker-alpha}).
\end{remark}
It is well known that $T_\alp$ has a meromorphic continuation in $\alp\in \CC$.
For $\alp_0\in \CC$ we will denote by $S_{\alp_0}$ the
first non-zero coefficient in the decomposition of the meromorphic function $T_\alp$ near $\alp_0$, namely
$$T_\alp=(\alp-\alp_0)^k\cdot( S_{\alp_0} +O(\alp-\alp_0))\mbox{ as } \alp\to \alp_0, \, S_{\alp_0}\ne 0.$$
Thus $$S_{\alp_0}\colon C^\infty(Gr_i(V))\to C^\infty(Gr_i(V))$$
is defined and non-zero for any $\alp_0\in \CC$, and coincides with $T_{\alp_0}$ for all but countably many $\alp_0$
(necessarily $Re(\alp_0)\leq -1$ for exceptional values of $\alp_0$). $S_\alp$ will also be called the $\alp$-cosine transform and sometimes denoted $S_\alp^i$.
The geometric meaning of $S_{\alp_0}$ for the exceptional values of $\alp_0$ is investigated in \cite{AlGS}.

Obviously the $\alpha$-cosine transform commutes with the natural action of the orthogonal group $O(n)$.
It was observed in \cite{AB} for $\alp=1$ and in \cite{alesker-alpha} for general $\alp$ that the $\alp$-cosine transform
can be rewritten in such a way to commute with an action of the much larger full linear group $GL(n,\RR)$. In that language it is a map
\begin{eqnarray}\label{E:invar-transform}
C^\infty(Gr_{i}(\RR^n),L_\alp)\to C^\infty(Gr_{n-i}(\RR^n),M_\alp),
\end{eqnarray}
where $L_\alp,M_\alp$ are appropriately chosen $GL(n,\RR)$-equivariant line bundles over Grassmannians, and a choice of a Euclidean metric on $\RR^n$ induces
identification $Gr_{n-i}(\RR^n)\simeq Gr_i(\RR^n)$ by taking the orthogonal complement, and $O(n)$-equivariant identifications of $L_\alp$ and $M_\alp$
with trivial line bundles. In fact the $\alpha$-cosine transform turned out to be a special case of the standard representation theoretical construction called
intertwining integral. This opened the way to apply the infinite dimensional representation theory of the group $GL(n,\RR)$. In that language we have $$C^\infty(Gr_{i}(\RR^n),L_\alp)\simeq \nu_i^{(n-i)/2} \times \nu_{n-i}^{-\alpha-i/2} , \quad C^\infty(Gr_{n-i}(\RR^n),M_\alp)\simeq \nu_{n-i}^{-\alpha-i/2}  \times\nu_i^{(n-i)/2} ,$$
and $S_{\alp}$ is the standard intertwining operator
$$S^{i}_{\alp}:\nu_i^{(n-i)/2} \times \nu_{n-i}^{-\alpha-i/2}  \to \nu_{n-i}^{-\alpha-i/2}  \times\nu_{i}^{(n-i)/2} .$$


Thus Theorem \ref{thm:irr} and Corollary \ref{cor:IntOp}   immediately imply the following corollary.
\begin{corollary}\label{Cor-main-corollary}
\begin{enumerate}
\item $S^{i}_\alp$ is not invertible if and only if there exists a non-negative integer $k$ such that one of the following holds:\\
$(a)\, \alp=2k, \quad (b) \, \alp=-n-2k,$\\
$(c)\, r>1 \text { and }\alp =1-r+k, \quad (d)\, r>1 \text { and } \alp=  r-n-1-k.$

\item \Dima{The range of $S^{i}_\alp$ is the intersection of all closed $G$-invariant non-zero subspaces of $\nu_{n-i}^{-\alpha-i/2}  \times\nu_{i}^{(n-i)/2}$.}

\item If $S^{i}_\alp$ is invertible then its inverse is a scalar multiple of $S_{-n-\alp}^{n-i}$.
\end{enumerate}
\end{corollary}

\subsection{Related works}\label{subsec:known}

In the archimedean cases Theorem \ref{thm:irr} was   proven earlier in \cite{HL}, the non-archimedean case follows from \cite{Zel}.
In fact \cite{HL} computes the length of $\chi_1\times\chi_2$ in the archimedean cases.
For $F=\R$ and $\chi=\nu^{s}$, where $|s|+n/2$ is not an
integer, or is bigger than $n-1$  Theorem \ref{thm:mon}  was proven in \cite{alesker-alpha}. For $F\in \{\R,\C\}$ and $p_1\geq p_2$ a similar  theorem was proven in \cite{HL}, where the irreducible constituents are characterized by their $K$-types, as opposed to the rank (or the Gelfand-Kirillov dimension). The assumption $p_1\geq p_2$ can be removed using the Cartan involution $g\mapsto \,^tg^{-1}$.
For $F=\R, \, n=2r$ this theorem was proven in \cite{BSS}.

\subsection{Acknowledgements} The main ideas of this paper appeared in connection to the work \cite{GS-Int}. We thank Siddhartha Sahi for allowing us to use them in this work. We thank Semyon Alesker for explaining us the connections to integral geometry, the anonymous referee for useful remarks, and Michal Zydor for translating the abstract into French.

\section{Preliminaries}
\label{sec:prel}

In this section we recall some of the necessary background for the benefit of the reader.

\subsection{Degenerate principal series}

Let $G$ be a reductive algebraic group over an arbitrary local field $F$. In this
section we discuss some basic properties of the induced representation
$I\left(  P,\chi\right)  $ on $C^{\infty}\left(  G,P,\chi\right)  $ as in
(\ref{=CG}). For detailed proofs we refer the reader to \cite{BW,Wal1}
or to other standard texts on representation theory.

Let $\mathcal{E}^{\prime}\left(  G\right)  $ denote the set of compactly
supported distributions on $G$, regarded as a left and right $G$-module as
usual via the pairing $\left\langle \cdot,\cdot\right\rangle :\mathcal{E}%
^{\prime}\left(  G\right)  \times C^{\infty}\left(  G\right)  \rightarrow
\mathbb{C}$.

\begin{lemma}
\label{lem: delta1} Let $\varepsilon\in\mathcal{E}^{\prime}\left(  G\right)  $
denote evaluation at $1\in G$, then we have
\[
\left\langle R_{p^{-1}}\varepsilon,f\right\rangle =\chi\left(  p\right)
\Delta_{P}^{1/2}\left(  p\right)  \left\langle \varepsilon,f\right\rangle
\text{ for all }p\in P,f\in C^{\infty}\left(  G,P,\chi\right)
\]

\end{lemma}

\begin{proof}
Indeed both sides are equal to $f\left(  p\right)  $.
\end{proof}

\begin{lemma}[{\cite[V.5.2.4]{Wal1}}]
\label{lem:duality} The representations $I\left(P,\chi\right)$ and
$I\left(P,\chi^{-1}\right)$ are contragredient.
\end{lemma}

Let $\bar{P}$ denote the parabolic subgroup opposite to $P$. Then the
characters of $P$ and $\bar{P}$ can be identified with those of the common
Levi subgroup $L=P\cap\bar{P}.$

\begin{proposition}[\cite{KnSt,Wald}]
\label{Prop:KnSt}There is a nonzero intertwining operator $I\left(
P,\chi\right)  \rightarrow$ $I\left(  \bar{P},\chi\right)  .$
\end{proposition}

\begin{lemma}
\label{lem:mac} If $K$ is a closed Lie subgroup (or an $l$-subgroup) of $G$ such that $PK=G$ then the
restriction $I(P,\pi)|_{K}$ is isomorphic to the induced representation $Ind_{P\cap K}^K(\pi|_{P\cap K})$. In
particular, the space $I(P,\pi)^{K}$ of $K$-fixed vectors is isomorphic to
$\pi^{P\cap K}$.

\end{lemma}

The isomorphism is given just by restricting the functions to $K$.
We will use this lemma for the maximal compact subgroup $K\subset G$ which in the non-Archimedean case is (up to conjugation) the subgroup $GL(n,\cO)$, where $\cO$ is the ring of integers of $F$. For $F=\R, \, K$ is the orthogonal group $O(n)$ and for $F=\C, \, K=U(n)$.

In both Archimedean cases,  $P \cap K$ is a symmetric subgroup of $K$ and we will use the following well-known lemma.

\begin{lemma}
\label{lem:SymGel} If $G$ is a connected and compact Lie group, $M$ is a symmetric
subgroup of $G$ and $\chi$ is a character of $M$ then $Ind_M^G(\chi)$ is a
multiplicity-free representation.
\end{lemma}
\begin{proof}
Let $V$ be an irreducible representation of $G$. Since $G$ is compact, $V$ is  finite-dimensional. The multiplicity of $V$ in $Ind_M^G(\chi)$ equals the dimension of $(V^*)^{M,\chi}$ (see e.g. Lemma 2.6 below). It is well known that there exists a Borel subalgebra $\mathfrak{b}$ of the Lie algebra $\mathfrak{g}$ of $G$ such that $\mathfrak{g}=\mathfrak{m}+\mathfrak{b}$, where $\mathfrak{m}$ is the Lie algebra of $M$. Let $v$ be a highest weight vector of $V$ and $\psi$ be its weight.

Let us show that any functional $\xi \in (V^*)^{M,\chi}$ is uniquely determined by its value on $v$.
Indeed, since $V$ is irreducible and $G$ is connected, any $w\in V$ equals $uv$, for some $u$ in the universal enveloping algebra $\mathcal{U}(\mathfrak{g})$.
Since $\mathfrak{g}=\mathfrak{m}+\mathfrak{b}$, we can present $u$ as $u=\sum_i m_i b_i$, where $m_i\in \mathcal{U}(\mathfrak{m})$ and $b_i\in \mathcal{U}(\mathfrak{b})$. Thus
$$\langle \xi , w \rangle =
\langle \xi , uv \rangle=\sum\langle \xi m_i , b_i v\rangle =
(\sum\chi(m_i)\psi(b_i))\langle \xi , v\rangle.$$
\end{proof}

\subsection{Finite dimensional representations}

Let $\left(  \phi,V\right)  $ be an irreducible finite dimensional
representation of a reductive group $G$. We are interested in the possibility
of realizing $\phi$ as a submodule or quotient of some $I\left(
P,\chi\right)  $, which we denote by $\phi\hookrightarrow I\left(
P,\chi\right)  $ and $I\left(  P,\chi\right)  \twoheadrightarrow\phi$
respectively. 
%
Let $\left(  \phi^{\ast},V^{\ast}\right)  $ be the contragredient
representation of $\left(  \phi,V\right)  $.
From  Lemmas
\ref{lem: delta1} and \ref{lem:duality} we obtain \Dima{the following lemma.}

\begin{lemma}[Frobenius reciprocity]
\label{lem:FinSubFrob} We have $\phi\hookrightarrow I\left(  P,\chi\right)  $(resp. $I\left(  P,\chi\right)  \twoheadrightarrow\phi$) if and only if $\left(  V^{\ast
}\right)  ^{P,\chi^{-1}\Delta_{P}^{-1/2}}$ (resp. $V^{P,\chi\Delta_{P}^{-1/2}%
}$) is nonzero. 
\end{lemma}
\begin{proof}
\Dima{Given a non-zero $l\in \left(  V^{\ast
}\right)  ^{P,\chi^{-1}\Delta_{P}^{-1/2}}$ we define a morphism $\phi\to I\left(  P,\chi\right)$ by $v \mapsto f$ with $f(g)=l(\phi(g^{-1})v)$. This map is an embedding since $\phi$ is irreducible. Conversely, an embedding  $\phi\hookrightarrow I\left(
P,\chi\right)  $ defines a functional on $\phi$ by composition with $\varepsilon\in\mathcal{E}^{\prime}\left(  G\right)  $. Dualizing $\phi$ and  $I\left(
P,\chi\right)  $ we obtain the statement on $I\left(  P,\chi\right)  \twoheadrightarrow\phi$.}
\end{proof}


\begin{cor}
\label{cor:NAFinSub} Let $\psi$ be a character of $G_{n}$. Then
$\dim\operatorname{Hom}(\psi,\chi\times 1)=1$ if and only if $\psi=\nu^{-p_{1}/2}$ and $\chi = \nu^{-n/2}$. Otherwise  $\operatorname{Hom}(\psi, \chi\times 1)=0$.
\end{cor}

From highest weight theory we also obtain

\begin{cor}
\label{cor:AFinSub} $\chi \times 1$ has a finite-dimensional
submodule if and only if either

\begin{enumerate}
[(i)]
\item \label{it:Finpadic}  $\chi=\nu^{-n/2}$.

\item \label{it:FinR}$F=\mathbb{R},$ and $\chi = \eps^{k} \nu^{-k-n/2}$ for some  non-negative
integer $k$.

\item \label{it:FinC}$F=\mathbb{C}$ and $\chi = \alpha^{(l-k)/2} \nu^{-(k+l+n)/2}$ for some  non-negative
integers $k,l$ of the same parity.
\end{enumerate}
\end{cor}

\subsection{Derivatives}

\label{sec:der}


If $\chi$ is a character of $G_{p}$ with $p>0$, we write $\chi^{\prime}$ for
its restriction to $G_{p-1}$.


We will prove the main results  by induction on $n$,
using results from \cite{BZ-Induced, AGS}. We refer the reader to those papers
for the notion of \emph{depth} for an admissible representation of $G_{n},$
and for the definition of the functor $\Phi$ which maps admissible
representations of $G_{n}$ of  depth $\leq2$ to admissible
representations of $G_{n-2}$. In \cite{BZ-Induced} this functor is denoted $\Phi^-$.

\begin{proposition}\label{prop:der}
(\cite{BZ-Induced, AGS})

\begin{enumerate}
\item $\Phi$ is an exact functor and $\Phi(\chi_{1}\times\chi_{2})=\chi
_{1}^{\prime}\times\chi_{2}^{\prime}$.

\item Every subquotient of $\chi_{i}\times\chi_{j}$ has depth $\leq2.$

\item If $\pi$ has depth $1$ then $\pi$ is finite dimensional and $\Phi(\pi)=0$.

\item If $\pi$ has depth $2$ then $\Phi\left(  \pi\right)  \neq0.$
\end{enumerate}
\end{proposition}

If $p_{1}=1$ (resp. $p_2=1$) then $\chi
_{1}^{\prime}\times\chi_{2}^{\prime}$ means just $\chi
_{2}^{\prime}$ (resp. $\chi_{1}^{\prime}$).


%
%
%

\subsection{Rank} \label{subsec:rank}

The rank of representations of $GL_n(F)$ for $F$ of  zero characteristic was defined in \cite{GSRank} using the notion of wave-front set, which is a certain union of  coadjoint nilpotent orbits. The rank is defined to be the maximal among the ranks of all the matrices in the  wave-front set.
For archimedean $F$ the dimension of the wave-front set equals \Dima{twice} the Gelfand-Kirillov dimension.

By Jordan's theorem the nilpotent orbits in $\mathfrak{gl}_n$ are given by partitions of $n$, i.e. the sizes of Jordan blocks. By e.g.
\cite[Theorem 2]{BB} the wave-front set of $\chi_{1}\times\chi_{2}$  is the closure of the nilpotent orbit given by the nilpotent orbit with $r$ blocks of size $2$ and $n-r$ blocks of size $1$. We will shortly denote such a partition by $2^r1^{n-r}$. Thus,  $rank(\chi_{1}\times\chi_{2})=r$.
The closure ordering on the orbits in the closure of $2^r1^{n-r}$  is linear, and thus the wave-front set of any subquotient $\pi$ of $\chi_{1}\times\chi_{2}$  is the closure of the nilpotent orbit given by the partition $2^{rank(\pi)}1^{n-rank(\pi)}$.
The dimension of this orbit is $2rank(\pi)(n-rank(\pi))$.

Since in this paper we want to consider fields of arbitrary characteristic we give the following definition of rank.

\begin{defn}\label{def:rank}
For a subquotient $\pi$ of $\chi_{1}\times\chi_{2}$ we define the rank of
$\pi$ to be the number $k$ such that $\Phi^{k}(\pi)$ is finite-dimensional and non-zero. In
particular, if $\pi$ is finite-dimensional then $rank(\pi)=0$ and if $\pi
=\chi_{1}\times\chi_{2}$ then $rank(\pi)=\min(p_{1},p_{2})$.
\end{defn}

For $F$ of  zero characteristic this definition coincides with the one given in \cite{GSRank}. This follows from \cite[Theorem 5.0.4]{GS-Gen} for Archimedean $F$ and from \cite[Theorem B]{GSRank} for non-Archimedean $F$.

\subsection{Infinitesimal characters}
\label{sec:InfChar}

In this subsection we assume that $F$ is Archimedean.
%
%
We denote the Lie algebra of $G$ by $\mathfrak{g}%
_{0}$ and its complexification by $\mathfrak{g}$, and we adopt similar
terminology for subgroups of $G$ and their Lie algebras. Let $\mathcal{Z}%
=\mathcal{Z}\left(  \mathfrak{g}\right)  $ be the center of the enveloping
algebra $\mathcal{U}\left(  \mathfrak{g}\right)  $ and let $\mathcal{Z}%
^{\prime}$ be the set of its multiplicative characters. We say that $\pi\in
Rep\left(  G\right)  $ has infinitesimal character $\xi\in Z^{\prime}$ if for
any $z\in\mathcal{Z}, \,\pi\left(  z\right)  $ acts by the scalar $\xi\left(
z\right)  $. If $\pi$ is irreducible then it has an infinitesimal character.

\begin{lemma}
If $\pi,\pi^{\prime}\in Rep\left(  G\right)  $ have infinitesimal characters
$\xi,\xi^{\prime}$, then Hom$_{G}\left(  \pi,\pi^{\prime}\right)  =0$ unless
$\xi=\xi^{\prime}.$
\end{lemma}


Fix a Cartan subalgebra and a Borel sublagebra $\mathfrak{h\subset b\subset
g}$. Let $W=W\left(  \mathfrak{g,h}\right)  $ be the corresponding Weyl group
and let $\rho$ be the half-sum of roots of $\mathfrak{h}$ in $\mathfrak{b}$.

\begin{theorem}
(Harish-Chandra) There is an isomorphism $\mu\mapsto\xi_{\mu}:\mathfrak{h}%
^{\ast}/W\rightarrow\mathcal{Z}^{\prime}$.
\end{theorem}

We often say that a representation has infinitesimal character $\mu$ instead
of $\xi_{\mu}$.

\begin{lemma}
If $\phi$ is a finite dimensional representation with highest weight $\mu
\in\mathfrak{h}^{\ast}$ then $\phi$ has infinitesimal character $\mu+\rho$.
\end{lemma}

If $P=LN$ is a parabolic subgroup, then we may choose $\mathfrak{h}$ so that
$\mathfrak{h\subset l}$; then $\mathfrak{h}$ is a Cartan subalgebra of
$\mathfrak{l}$ as well.

\begin{lemma}\label{InfCharInd}
If $\pi\in Rep\left(  L\right)  $ has infinitesimal character $\mu
\in\mathfrak{h}^{\ast}$, then $I\left(  P,\pi\right)  $ has infinitesimal
character $\mu$ also.
\end{lemma}

\section{Finite-dimensional subquotients of $\chi \times  1$}

\label{sec:GL}


%
%
%
%
%
%
%

\subsection{Archimedean case}

\label{subsec:arch}





Consider first the case $F=\mathbb{R}$. By the Harish-Chandra isomorphism, the
characters of $\mathcal{Z}\left(  G_{n}\right)  $ can be identified with
$\mathbb{C}^{n}/S_{n}$ where $S_{n}$ is the symmetric group. Thus
infinitesimal characters can be thought of as \emph{multisets} of complex
numbers.

\begin{defn}
\label{def:seg} Let $\xi=\{z_{1},\dots,z_{n}\}$ be a multiset of complex
numbers with $\operatorname{Re}z_{i}\geq\operatorname{Re}z_{i+1}$. We say that
$\xi$ is a \emph{segment} if $z_{i}-z_{i+1}=1$ for all $i$, and a
\emph{generalized segment} if $z_{i}-z_{i+1}$ is an integer for all $i$.
For $s\in \C, \, p\in \N$ we denote the segment of length $p$ with center at $s$ by $\xi_{p}^s$.

\end{defn}

It is well-known that the infinitesimal character of a character $\chi=\eps^k\nu^s$ of $GL_p(\R)$ is $\xi_{p}^s$, and that the infinitesimal character
of an irreducible finite-dimensional
representation is a generalized segment. From Lemma \ref{InfCharInd} we obtain that the infinitesimal character of $\chi \times 1$ is the disjoint (multiset) union $\xi_{p_1}^0\sqcup\xi_{p_2}^s$.

%
%
%

For a representation of $GL_{n}(\mathbb{C})$ the infinitesimal character is a
pair of multisets $(\eta,\zeta),$ the infinitesimal character of $\chi=\alpha^k\nu^s$ is $(\xi_{p}^{s+k},\xi_{p}^{s-k})$, and the infinitesimal character of $\chi
\times1$ is
$(\xi_{p_1}^{s+k}\sqcup\xi_{p_2}^{0},\xi_{p_2}^{s-k}\sqcup\xi_{p_2}^{0})$.

%
%
%

Let us now consider the $K$-types of $\chi\times1$, for
$F\in\{\mathbb{R},\mathbb{C}\}$. Here, $K$ is the maximal compact subgroup of
$G_{n}$ given by  $K=O(n)$ for $F=\mathbb{R}$ and $K=U(n)$ for $F=\mathbb{C}$. We
denote by $K^{0}$ the connected component of $K$, thus $K^{0}=SO(n)$ for
$F=\mathbb{R}$ and $K^{0}=K$ for $F=\mathbb{C}$.

\begin{lemma}
\label{lem:signs}

\begin{enumerate}
[(i)]

\item \label{it:trivKtype} If $\chi\times1$ has the trivial K-type
then $\chi$ is a power of $\nu$.

\item \label{it:mult} The restriction of $\chi\times 1$ to $K^{0}$
is multiplicity free

\item \label{it:signs} Suppose $F=\mathbb{R}$, $\chi=\eps^k\nu^s$ and let $\psi=\eps^l\nu^t$ be a character of $G_{p_1}$, such that $\chi\times 1$ and $\psi\times 1$ have a common $K$-type. Then $k=l \,\, (mod \,\,2)$.
\end{enumerate}
\end{lemma}

\begin{proof}
Let $P:=P_{p_{1},p_{2}}, M:=K^{0}\cap P$ and
$M^{0}$ be the connected component of $M$. Consider the involution on $K^0$ given by conjugation by the block-diagonal matrix $\diag (Id_{p_1} , -Id_{p_2})$. The fixed points group of this involution is $M$. Thus both $M$ and $M^0$ are symmetric subgroups of $K^0$.

Let $\chi_K:=(\chi
\otimes1)|_{P\cap K}$.
By Lemma \ref{lem:mac} we have $(\chi\times 1)|_{K}=Ind_{P\cap K}^K(\chi_K)$
and $(\chi\times 1)|_{K^{0}}=Ind_{M}^{K^0}(\chi_K|_M)$.

(\ref{it:trivKtype}) follows now from Lemma \ref{lem:FinSubFrob} and
(\ref{it:mult}) from Lemma \ref{lem:SymGel}.

For (\ref{it:signs}) note that if $k$ and
$l$ have different parities then one of the
representations $(\chi\times 1)|_{K^{0}}$ and $(\psi\times
1)|_{K^{0}}$ is $Ind_M^{K^0}(1)$ and the other is $Ind_M^{K^0}(\eps\otimes 1)$. Thus, $(\chi\times 1)|_{K^{0}}\oplus(\psi
\times 1)|_{K^{0}}=Ind_{M^0}^{K^0}(1) $ is multiplicity free by Lemma
\ref{lem:SymGel}, and thus $\chi\times 1$ and $\psi\times 1$ cannot have a common $K$-type.
\end{proof}

\begin{cor}
\label{cor:UniqFinDimSQ} $\chi\times 1$ has at most one irreducible
finite-dimensional subquotient.

\end{cor}

\begin{proof}
Different irreducible finite-dimensional representations of the connected
component $G_{n}^{0}$ of $G_{n}$ have different infinitesimal characters. Thus
only one irreducible finite-dimensional representation of $G_{n}^{0}$ can
appear as a subquotient. It appears only once by Lemma \ref{lem:signs}%
(\ref{it:mult}).
\end{proof}

\subsection{Non-archimedean case}

\label{subsec:non-arch}

In this subsection we assume that $F$ is non-archimedean. Thus every
irreducible finite-dimensional representation of $G_{n}$ is a character. 

We need the next lemma, which follows from \cite[Theorem
2.9]{BZ-Induced}.

\begin{lem}
\label{lem:PrinJH} Let $B_{n}\subset G_{n}$ be the usual Borel subgroup,
$\phi_{1}$ and $\phi_{2}$ be characters of $B_{n}$ given by $n$-tuples  $X_{1}
\text{ and } X_{2}$ of characters of $F^{\times}$. Suppose that the principal
series $I(B_{n},\phi_{i})$ have a common constituent. Then $X_{1}$ is a
permutation of $X_{2}$.
\end{lem}

\begin{cor}
\label{cor:pre1dimpadic} If $\chi\times 1$ has a one-dimensional
subquotient $\psi$ then $\psi$ is unique and either $\chi=\nu^{-n/2}, \, \psi=\nu^{-p_2/2}$ and $\psi$ is a submodule or $\chi=\nu^{n/2}, \, \psi=\nu^{p_1/2}$ and $\psi$ is a quotient.
\end{cor}

\begin{proof}
Define $a_1,a_2,a_3$ to be characters of $F^\times$ such that
$\chi=a_1\circ\det$, $\psi=a_{3}\circ\det$ and $a_2=1$.  Let $p_{3}:=n$ and
\[
Y_{i}:=(a_{i}\nu^{-(p_{i}-1)/2},a_{i}\nu^{-(p_{1}-3)/2},\dots, a
_{i}\nu^{(p_{i}-1)/2}), X_{1}:=Y_{3} \text{ and }X_{2}:=(Y_{1},Y_{2}).
\]
Then $\chi \hookrightarrow I(B_{p_{1}},Y_{1})$, $1 \hookrightarrow I(B_{p_{2}},Y_{2})$ and $\psi \hookrightarrow I(B_{p_{3}},Y_{3})$. Thus $\chi\times
1 \hookrightarrow I(B_{n},X_{2})$ and thus $\psi$ is a joint
subquotient of $I(B_{n},X_{1})$ and $I(B_{n},X_{2})$. Thus $X_{1}$ is a
permutation of $X_{2}$, which implies that either $a_{3}=\nu^{-p_{2}/2}$ and $a_{1}=\nu^{-n/2}$ or $a_{3}=\nu^{p_{2}/2}$ and $a_{1}=\nu
^{n/2}$. In the first case we get $\chi=\nu^{-n/2}, \, \psi=\nu^{-p_2/2}$ and, by Corollary \ref{cor:NAFinSub} $\psi$ is a submodule, and in the second case $\chi=\nu^{n/2}, \, \psi=\nu^{p_1/2}$ and $\psi$ is a quotient.

The uniqueness of $\psi$ follows from the uniqueness of the $K$-fixed vector (Lemma \ref{lem:mac}).\end{proof}


\section{Proof of the main results}\label{sec:Pf}

Denote $\pi:=\chi \times 1$, $r:=\min(p_{1},p_{2})$.

\begin{proof}[Proof of Theorem \ref{thm:irr}]
We prove the statement by induction on $r$. For the base of
the induction, note that if $r=1$ then $\Phi(\pi)$ is
1-dimensional and thus, by Proposition \ref{prop:der}, $\pi$ is reducible if and only if it has either a
finite-dimensional submodule or a finite-dimensional quotient. By Corollary \ref{cor:AFinSub} and Lemma \ref{lem:duality}, this is equivalent to the conditions
of the proposition.

Note that if one of the conditions holds for $\chi^{\prime}$ and $n-2$ then it holds for $\chi$ and $n$. Conversely, if one of
the conditions holds for $\chi$ (and $n$) then either it holds for
$\chi^{\prime}$ (and $n-2$) or $\pi=\chi\times 1$ has a
finite-dimensional submodule or quotient.

Now suppose $\pi$ is irreducible but one of the conditions
holds. Since we have standard intertwining operators between $\pi$ and
$1\times \chi$, this implies that $\pi$ is a direct
summand of $1\times\chi$. Applying $\Phi$ we get that $\Phi(\pi)=\chi
^{\prime}\times 1$ is a direct summand of $1 \times \chi^{\prime}$. By the induction hypothesis, $\chi^{\prime}\times 1$ is reducible. Applying $\Phi$ again and again if
needed, we can assume (Corollary \ref{cor:AFinSub}  and Lemma \ref{lem:duality}) that $\chi^{\prime}\times1$ has a
finite-dimensional submodule or a finite-dimensional quotient. Then Corollary \ref{cor:AFinSub} and Lemma \ref{lem:duality} imply that $1\times\chi^{\prime}$ cannot have a finite-dimensional constituent from the
same side, contradicting having $\chi^{\prime}\times 1 $ as
a direct summand.

Now, suppose $\pi$ is reducible. If it has more than one
infinite-dimensional constituent then $\chi^{\prime}\times 1$ is reducible and thus, by the induction hypothesis, one of the
conditions holds. Otherwise, $\pi$ has either a finite-dimensional submodule
or a finite-dimensional quotient. Dualizing if needed we can assume that $\pi$ has
a submodule and then Corollary \ref{cor:AFinSub} implies that
one of the conditions holds.
\end{proof}

For the proof of Theorem \ref{thm:mon} we will need the following lemma.

\begin{lem}
\label{lem:AppFinDim}If $\chi_{}\times 1$ has a finite-dimensional
irreducible subquotient $\phi$, then either $\phi$ is a quotient and $s(\chi)>0$ or $\phi$ is
a submodule and $s(\chi)<0$.
\end{lem}

\begin{proof}
For non-archimedean $F$ this follows from Corollary \ref{cor:pre1dimpadic}, and thus
we assume that $F$ is archimedean and the infinitesimal character $\xi$ is a
generalized segment (if $F=\mathbb{R}$) or a pair of generalized segments (if
$F=\mathbb{C)}$. Assume (dualizing if needed) that $s(\chi)< 0$. If $F=\mathbb{C}$ then Theorem \ref{thm:irr} and Corollary
\ref{cor:AFinSub} imply that $\chi\times 1$ has a finite-dimensional
submodule $\phi^{\prime}$. If $F=\mathbb{R}$ we let $\chi=\eps^k\nu^s$ and note that there exists a character $\psi=\eps^l\nu^t$ of $G_{p_1}$ such that $\phi$ can be realized as a submodule of
$\psi\times 1$. This implies that $s+n/2,t+n/2,
\text{ and }l$ have the same parity. By Lemma
\ref{lem:signs}(\ref{it:signs}) $k$ also has the
same parity, and thus, by Corollary \ref{cor:AFinSub}, $\chi\times1$
has a finite-dimensional submodule $\phi^{\prime}$. By Corollary
\ref{cor:UniqFinDimSQ} $\phi^{\prime}$ must coincide with $\phi$.
\end{proof}

\begin{proof}
[Proof of Theorem \ref{thm:mon}]We can assume that $s(\chi)>0$. Then it follows from the previous lemma that $\pi$ has
at most one irreducible finite-dimensional subquotient $\phi$, which is
necessary a quotient. Now, let $\tau_{1},\dots,\tau_{k}$ be the other
irreducible constituents. Then $\Phi(\tau_{1}),\dots,\Phi(\tau_{k})$ are
non-zero (not necessary irreducible) constituents of $\chi_{1}^{\prime}%
\times\chi_{2}^{\prime}$. Arguing by induction, we can assume that
$rank(\Phi(\tau_{1}))>\cdots> rank(\Phi(\tau_{k}))$ and thus $rank(\tau
_{1})>\cdots> rank(\tau_{k})>0=rank(\phi)$.
\end{proof}

\begin{lem}
\label{lem:padic} If $F$ is non-Archimedean then $\chi\times 1$  has length
1 or 2.
\end{lem}
\begin{proof}
We prove by induction on $r$. The base case
$r=0$ is trivial.

If $\chi\neq\nu^{\pm n/2}$ then, by
Corollary \ref{cor:pre1dimpadic}, $\chi\times 1$ has no
finite-dimensional constituents. Applying the functor $\Phi$ and using Proposition
\ref{prop:der} we obtain $\operatorname{length}(\chi\times 1)
\leq\operatorname{length}(\chi^{\prime}\times 1)$. The
induction hypotheses implies that $\operatorname{length}(\chi^{\prime}%
\times 1) \leq2$.

Now consider the case $\chi=\nu^{\pm n/2}$. Theorem \ref{thm:irr} implies that $\chi'\times 1$ is irreducible and thus $\chi\times 1$ has a unique
infinite-dimensional irreducible constituent. By Corollary
\ref{cor:pre1dimpadic} it also has a unique one-dimensional constituent and thus has length two.
\end{proof}

\Dima{
\begin{proof}
[Proof of Corollary \ref{cor:IntOp}]
From Theorem \ref{thm:irr} we see that if $\chi\times 1$ is irreducible then so is $1\times \chi$ and thus, by Schur's first lemma, $I_{\chi}$ is an isomorphism.

Assume now that $\chi\times 1$ (and thus $1\times \chi$) are reducible.
 Note that $1\times \chi=\chi\otimes(\chi^{-1}\times 1)$ and $s(\chi^{-1})=-s(\chi)$. Thus by Theorem \ref{thm:irr} the irreducible constituents of $\chi\times 1$ and of $1\times \chi$ appear in opposite orders of rank. Let $\sigma$ denote the image of $I_{\chi}$. Note that $\sigma$ is a quotient $\chi\times 1$ and a subrepresentation of $1\times \chi$. Thus the irreducible constituents
of $\sigma$ appear on one hand in descending order of rank and on the other hand in ascending order of rank. This implies that there is only one constituent, {\it i.e.} $\sigma$ is irreducible. This implies in particular that $I_{\chi}$ is not an isomorphism.
\end{proof}
}



\end{document}